\documentclass[12pt,reqno]{amsart}
\usepackage[a4paper,bindingoffset=0.1in,left=1.2in,right=1.3in,top=1.4in,bottom=1.5in,footskip=.25in]{geometry}
\usepackage{indentfirst,amssymb,amsfonts,amsmath,amsthm}    
\usepackage{newtxtext,newtxmath}
\usepackage{setspace}
\usepackage{times}
\usepackage[utf8]{inputenc}
\usepackage{stackengine}
\usepackage[T1]{fontenc}
\usepackage{verbatim}
\usepackage{hyperref}
\hypersetup{colorlinks=true, linkcolor=blue,citecolor=blue, urlcolor=blue}
\usepackage{mathtools}
\numberwithin{equation}{section}

\newtheorem{thm}{Theorem}[section]

\newtheorem{prop}{Proposition}[section]
\newtheorem{lem}{Lemma}[section]

\newtheorem{rema}{Remark}[section]
\newtheorem{coro}{Corollary}[section]
\newtheorem{notation}{Notation}
\newtheorem*{theorem_1}{Theorem 1}
\newtheorem*{theorem_2}{Theorem 2}

\newtheorem*{theorem_A}{Theorem A}
\newtheorem*{theorem_B}{Theorem B}
\newtheorem*{theorem_C}{Theorem C}
\newtheorem*{theorem_D}{Theorem D}

\DeclareMathOperator{\rk}{rk}

\DeclareMathOperator{\alt}{Alt}
\DeclareMathOperator{\tr}{Tr}
\DeclareMathOperator{\gal}{Gal}
\DeclareMathOperator{\mo}{mod}
\DeclareMathOperator{\ord}{ord}
\DeclareMathOperator{\gc}{gcd}

\makeatletter
\g@addto@macro{\endabstract}{\@setabstract}
\newcommand{\authorfootnotes}{\renewcommand\thefootnote{\@fnsymbol\c@footnote}}%
\makeatother

\begin{document}
\setcounter{page}{1} 

\baselineskip .65cm 
\pagenumbering{arabic}
\title{ Constant rank subspaces of alternating bilinear forms from Galois Theory }
\author [Ashish Gupta and~Sugata Mandal]{Ashish Gupta, Sugata Mandal}

\address{Ashish Gupta, Department of Mathematics\\
Ramakrishna Mission Vivekananda Educational and Research Institute (Belur Campus) \\
Howrah, WB 711202\\
India}
\email{a0gupt@gmail.com \thanks{}}
\address{Sugata Mandal, Department of Mathematics\\
Ramakrishna Mission Vivekananda Educational and Research Institute (Belur Campus) \\
Howrah, WB 711202\\
India}
\email{gmandal1961@gmail.com\thanks{}}

\maketitle
\begin{abstract}
Let $L/K$ be a cyclic extension of degree $n = 2m$. It is known that the  space $\alt_K(L)$ of alternating $K$-bilinear forms (skew-forms) on $L$ decomposes into a direct sum of $K$-subspaces $A^{\sigma^i}$ indexed by the elements of $\gal(L/K) = \langle \sigma \rangle$. It is also known that the components $A^{\sigma^i}$ can have nice constant-rank properties. We enhance and enrich these constant-rank results and show that the component $A^\sigma$ often decomposes directly into a sum of constant rank subspaces, that is, subspaces all of whose non-zero skew-forms have a fixed rank $r$. In particular, this is always true when $-1 \not \in L^2$. As a result we deduce a decomposition of $\alt_K(L)$ into subspaces of constant rank in several interesting situations. 
We also establish that a subspace of dimension $\frac{n}{2}$ all of whose nonzero skew-forms are non-degenerate can always be found in $A^{\sigma^i}$ where $\sigma^i$ has order divisible by $2$. 


\noindent 
\textbf{2020 Math. Subj. Class.}: 12F05, 12F10, 15A63
\end{abstract}
\section{Introduction}
Let $K$ be a field of characteristic other than two and $\alt_K (V)$ denote the space of all alternating bilinear forms (skew-forms) on a $K$-space $V$ of dimension $n$. Suppose $K$ admits a Galois extension $L$ of degree $n$. Taking the $n$-dimensional $K$-space $L$ as a model for $V$ it was shown in \cite{GQ09} that ideas from Galois Theory can be fruitfully applied for studying skew-forms on $V$. Notably, this approach sheds light on the subspaces of $\alt_K (V)$ whose nonzero skew-forms all have the same rank equal to $k$, say. Such ``$k$-subspaces" besides being interesting in their own right  play an important role in coding theory (see \cite{PD75},\cite{PD78}). Of particular importance are  the $n$-subspaces of $\alt_K(V)$, that is, subspaces all of whose nonzero skew forms are non-degenerate.   

Replacing $V$ by the $K$-space $L$, we begin with some definitions and facts given in \cite[Lemma 2]{GQ09}. 
For each $\sigma \in G : = \gal(L/K)$ and $b \in L$ we may define the skew-form 
\begin{equation} \label{defn-sk-frm}
f_{b,\sigma}(x,y) = \tr^L_K(b(x\sigma(y)-\sigma(x)y)), \qquad \forall x, y \in L.  
\end{equation}

where $\tr^L_K : L \rightarrow K$ is the Galois-theoretic trace map defined by \[ \tr^L_K(a) = \sum_{\sigma \in \gal(L/K)} \sigma(a), \quad \forall a \in L. \]

With each $\sigma \in G$ we can thus associate a subspace $A^\sigma$ of $\alt_K(L)$ defined as $A^\sigma : = \{f_{b, \sigma} : b \in L \}$. Each $ A^{\sigma}$ has dimension $n$ unless  $\sigma$  has order $2$ (see \emph{\cite[Theorem 1]{GQ09}}).
It  was shown  in \cite{GQ09}  that $\alt(L)$ decomposes as a direct sum of the spaces $A^\sigma$ with $\sigma$ ranging over the elements of the Galois group $G$ (see Theorems 1 and 2 below).   

Let $ \ord(\sigma) $ denote the order of $ \sigma \in G $. Interestingly, for $n$ odd, each $A^\sigma$ is an $ n - n/{\ord(\sigma)}$-subspace (Theorem 1). However when $n$ is even the situation is less clear as in this case we only know that the subspace $A^\sigma$ has a constant rank property only when  $\sigma$  is either an involution or else it has odd order  (see Section 2).  When $\sigma$ has even order it is only known that a skew form $f_{b, \sigma} \in A^{\sigma}$ may have rank either $n$  or $n - 2n/\ord(\sigma) $  and that both of these values are attained as ranks of suitable skew forms in $A^{\sigma}$. 
We study this last case more closely here  and show that there are constant-rank subspaces in $A^{\sigma}$.  In fact,  $A^{\sigma}$ always has an $n$-subspace of dimension $\frac{n}{2}$ and moreover decomposes as a direct sum of $k$-subspaces for suitable $k$ (see Theorems A-D).  



\begin{theorem_1}[\cite{GQ09})]
    \label{odd decomposition}
    Suppose that $ n = [L : K] $ is odd and the Galois group $G = \{1,\sigma_1,\cdots,\sigma_{m}, \sigma_{1}^{-1},\cdots,\sigma_{m}^{-1}\}$ where  $ m = (n - 1)/2 $. Then there is a direct  decomposition 
	\begin{equation}\label{odd direct decom}
	 \alt_K(L)=A^{1}\oplus A^{2}\oplus \cdots\oplus A^{m}, 
	\end{equation} 
	where $ A^{ i} := A^{\sigma_i}$ has dimension $n$ ($1 \le i \le m$). Moreover, if $\ord(\sigma_{i})$ = $2r_i +1$, the non zero skew-forms in $A^i$ all have rank $n - \frac{n}{2r_i +1} $. 
\end{theorem_1}
\begin{theorem_2}
\emph{(\cite{GQ09})}
Suppose that $ n = [L : K] $ is even and the Galois group $$G = \{1, \tau_1,\cdots,\tau_{k},\sigma_1,\cdots,\sigma_{m}, \sigma_{1}^{-1},\cdots,\sigma_{m}^{-1}\},$$ where $ \{\tau_{1},\tau_{2} ,\cdots,\tau_{k} \}$ are the involutions of $ G $, then there is a direct  decomposition 
	\begin{equation}\label{ev decompose}
	\alt_K(L)= B^{1}
\oplus B^{2} \oplus \cdots\oplus B^{k}\oplus A^{1} \oplus A^{2} \oplus \cdots \oplus A^{m}.    
	\end{equation}
	 where $B^{i}: =A^{\tau_i}$ is an $n$-subspace of 
  dimension $n/2$ for all $1 \le i \le k$ and   $A^{j}: = A^{\sigma_j}$ ($1 \le j \le m$) has dimension $n$. Moreover if $\ord(\sigma_i)$ is odd then $A^{\sigma}$ is an $n - n/{\ord(\sigma_i)}$-subspace of dimension $n$. 
\end{theorem_2}
   If $ L/K$ is cyclic Galois extension of degree $ n $  with $G= \gal(L/K) = \langle \sigma \rangle$ we define $A^i := A^{\sigma^{i}}$. Thus $ A^{i}= \{f_{b,\sigma^i}: b \in L\}$. If  $n$ is even  then there is a unique involution $ \tau_1 = \sigma^{n/2}$ and in this case we denote $B^1 : =A^{\tau_1}=\{f_{b,\sigma^{n/2}}: b\in L\}$. Then the decomposition \eqref{ev decompose} becomes 
\begin{equation}\label{ dcmpstn for cyclic}
\alt_K(L)=  B^{1} \oplus A^{1}\oplus A^{2}\oplus \cdots\oplus A^{m},
\end{equation}
   \begin{theorem_A} \label{odd dcmposn for any field}
   Let $K$ be a field and $n = 2k$, where  $ k\geq 1$ is odd.  Let $L$ be any cyclic  extension of $K$ of degree $n$ with Galois group $G = \langle \sigma \rangle$. 
   Then
   \begin{equation} \label{odd dcmposn for any field eq}
 A^1 = \mathcal{U}_{1} \oplus \mathcal{V}_{1},
 \end{equation}
 where    $ \mathcal{U}_{1}$ is an $n$-subspace  of dimension $k $ and $\mathcal{V}_{1}$ is an $ (n-2)$-subspace of dimension $k$.
  \end{theorem_A}

In view of Theorem A in following theorems we focus on the case where $ n$ is divisible by $ 4 $.  

\begin{theorem_B} 
  Suppose $ n = 2^{\alpha} k $ where $ \alpha \geq 2$ and $k$ is odd. Let  $ K $ be an algebraic number field such that $ -1$ is not a square in $ K$. Then there exists a cyclic  extension $L$ of $ K $ of degree $ n $ with the Galois group $G = \langle \sigma \rangle$ such that       
   \begin{equation}\label{decmpsn A1 algebraic}    
A^{1} =  \mathcal{E}_1 \oplus  \cdots  \oplus \mathcal{E}_{\alpha - 1} \oplus \mathcal{V}_1 \oplus \mathcal{V}_2 ,
  \end{equation}
   where
    \begin{itemize}
       \item[(i)] $\mathcal{E}_{i}$ is an  $n$-subspace of dimension $n/2^i$ for $ 1 \leq i \leq \alpha - 1$,
        \item[(ii)] $\mathcal{V}_j$ is an $( n-2)$-subspace of dimension $ k $ for $ 1 \leq j \leq 2$. 
    \end{itemize}
  \end{theorem_B}  \begin{theorem_C}\label{th for finite field}
  Let $ K $  be a finite field with $ q$ elements such that $ -1$ is not a square in $ K$. Let  $ q + 1 = 2^a l$ (l odd) where $a \geq 1$ and  $ n = 2^{\alpha}k$ (k odd) where  $ \alpha \geq 2$. Suppose   $ L $ is a cyclic  extension of $ K $ of degree $ n $ with $\gal(L/K)=\langle\sigma_f\rangle$ where $\sigma_f$ is the  Frobenius map of $L$ defined by  $\sigma_f: b \rightarrow b^q$. 
         \begin{itemize}
             \item[(1)] 
 If $\alpha \leq a+1 $ then  
   \begin{equation}\label{ A1 finite field}    
A^{1} = \mathcal{V}_1 \oplus \mathcal{V}_2 \oplus \mathcal{E}_1 \oplus \cdots \oplus \mathcal{E}_{\alpha - 1},
   \end{equation}
    where
    \begin{itemize}
        \item[(i)] $\mathcal{E}_{i}$ is an  $n$-subspace of dimension $n/2^i$ for $ 1 \leq i \leq \alpha - 1$,
        \item[(ii)]  $\mathcal{V}_j$ is an $ (n-2)$-subspace of dimension $ k $ for $ 1 \leq j \leq 2$. 
    \end{itemize} 
    \item[(2)] If $\alpha >a+1 $ and $ l=1$, that is, $ q = 2^a - 1$, then 
   \begin{equation}\label{decmpsn finite field2}          
A^{1} = \mathcal{V}_1 \oplus \mathcal{V}_2 \oplus \mathcal{E}_1 \oplus \cdots \oplus \mathcal{E}_{\alpha - 1},
   \end{equation}
   where
    \begin{itemize}
        \item[(i)]   $\mathcal{E}_{i}$ is an  $n$-subspace of dimension $n/2^i$ for $ 1 \leq i \leq a$ and  an $ (n-2)$-subspace of dimension $n/2^i$ for $ a+1 \leq i \leq \alpha -1$,
        \item[(ii)]  $\mathcal{V}_j$ is an $( n-2)$-subspace of dimension $ k $ for $ 1 \leq j \leq 2$. 
    \end{itemize}
         \end{itemize}
  \end{theorem_C}
  \begin{theorem_D}
   Let $p$ be a prime and $K = \mathbb{Q}_p$ be the $p$-adic  completion of $ \mathbb{Q}$ such that $ -1$ is not a square in $ K$. Let  $ p + 1 = 2^a l$ (l odd) where $a \geq 1$ and  $ n = 2^{\alpha}k$ (k odd) where  $ 2 \leq \alpha \leq a+1$. Then there exists a cyclic  extension   $ L $ of  $ K $ of degree $ n $ such that the decomposition \eqref{ A1 finite field}    holds.  
  \end{theorem_D}
\section{Skew forms and Galois extensions}\label{section2}
Retaining the notation of the previous section  we now collect some basic results from  \cite{GQ09} concerning the application of Galois theory to the study of some crucial properties of bilinear forms over $K$. In  the following $L/K$ is a (not necessarily cyclic) Galois extension and $1 \ne \sigma \in \gal{L/K}$ is arbitrary. 

\begin{lem}\emph{(\cite[Lemma 2]{GQ09})}
\label{GQ degeneracy condition}
Let $ f = f_{ b,\sigma} $ be an alternating bilinear form as defined above with $ b \neq 0 $ and let $F$ be the fixed
field of the automorphism $\sigma^{2}$. If $\sigma(b)b^{-1}$ is expressible in the form $ \sigma^{2}(c)c^{-1} $ for some $ c \in L^{\times}$ then 
$\rk(f_{ b,\sigma})  = n - n/[L:F]$. Otherwise $ \rk(f_{ b,\sigma}) = n $.
\end{lem}
\begin{lem}\emph{(\cite[Lemma 4]{GQ09})} \label{existance a non degenarate}
 Suppose that the automorphism $\sigma $ has even multiplicative order $ 2r $, say. Then there exist elements $ b \in L^{\times} $ such that the equation $ \sigma(b)b^{-1}  = \sigma^{2}(c)c^{-1} $ has no solution for all $ c\in L^{\times} $. 
\end{lem}
\begin{rema}\label{isomorphism}
    If $\sigma$ is not an involution then  the map $ b \rightarrow f_{b,\sigma}$ defines  an isomorphism of  $K$-spaces between $A^\sigma$ and $L$  \emph{\cite[Theorem 1]{GQ09}}.
\end{rema}

   \begin{lem}\emph{(\cite[Lemma 3]{GQ09})}
    \label{rank of f b sigma for odd case }  Suppose that the automorphism $ \sigma $ has odd multiplicative order $ 2r + 1 > 1 $, say. Then, if $ b \neq 0 $,
the rank of the skew-form $ f = f_{ b,\sigma } $ is $ n - n/{2r+1} $.
\end{lem}
\begin{lem}\emph{(\cite[Lemma 4]{GQ09})}
 \label{rank of f b sigma for even case } Suppose that the automorphism $\sigma$ has even multiplicative order $2r \geq 2$, say. Then, if $ b \neq 0 $,
the rank of the skew-form  $f=f_{b,\sigma} $ is either $ n-\frac{n}{r}$ or $ n $.
\end{lem}
\section{Preliminary results}
Our aim in this section is to establish certain facts which will be found useful in the subsequent sections and are also interesting in their own right.
   Recall that if $F$ is an intermediate subfield and $a \in L$ then the $L/F$-norm  
 $N_{L/F}(a)$ of $a $ is defined as $N_{L/F}(a) = \displaystyle\prod_{\theta \in \gal(L/F)} \theta(a)$.
 \begin{notation} \label{L_i}
\textbf{Throughout this section $L/K$ denotes a cyclic extension with Galois group $\gal(L/K) = \langle \sigma \rangle$}. For the sake of convenience in what follows we shall denote the subfield $L^{\langle \sigma^{i} \rangle}$ as $L_i$.
\end{notation}
 We begin by noting the following restatement of the degeneracy criterion Lemma \ref{GQ degeneracy condition}. 

\begin{prop}
\label{GMdegen-crit}
Let $b \in L$. Then the skew-form $f_{b,\sigma}$ is degenerate if and only if 
\begin{equation}\label{form1} 
N_{L/L_{2}} (\sigma(b)/b) = 1, 
\end{equation}
that is, $f_{b,\sigma}$ is degenerate if and only if \begin{equation} \label{form2}
N_{L/L_{2}} (b) =  b\sigma^{2}(b)\cdots \sigma^{n-2}(b) \in K. 
\end{equation}
\end{prop}

\begin{proof}
By Lemma \ref{GQ degeneracy condition}, the skew form $f_{b, \sigma}$ is degenerate if and only if $\sigma(b)/b = \sigma^2(c)/c$ for some $c \in L$. The first assertion is now clear in view of the Hilbert Theorem 90.
Moreover the condition $ N_{L/L_2 }(\sigma(b)/b) = 1 $ is easily seen to be equivalent to the product $b\sigma^{2}(b)\cdots \sigma^{n-2}(b) $ being $\sigma$-invariant. 

\end{proof}

Suppose that $\sigma^i$ is not an involution. By Lemma \ref{GQ degeneracy condition} the skew-form $f_{b,\sigma^i} \in A^i \subseteq \alt_K(L) $ is degenerate if and only if $\sigma^i(b)/b= \sigma^{2i}(c)/c$.  As $\sigma^{2i}$ is a generator for  $\gal(L/L_{2i})$,  in view of Hilbert Theorem 90, $f_{b,\sigma^i}$ is degenerate if and only if  $N_{L/L_{2i} }(\sigma^i(b)/b) = 1$. 
   A glance at Proposition  \ref{GMdegen-crit} above shows that this is precisely the condition for the skew-form $ \displaystyle f^{\sim}_{b,\sigma^{i}} \in \alt_{L_i}(L)$ defined by
    \[ f^{\sim}_{b,\sigma^{i}} = \tr^L_{L_i} (b(x\sigma(y)-\sigma(x)y)), \qquad \forall x, y \in L. \] 
   to be degenerate (we write $f^{\sim}_{b,\sigma^{i}}$ instead of $f^{}_{b,\sigma^{i}}$ to emphasize the fact that we are now considering $L$ as $L_i$-space). 
   
   Let us write
   $\displaystyle A^{\sim 1} := \{ \displaystyle f^{\sim}_{b,\sigma^{i}} \mid b \in L\}$.
   In view of Remark \ref{isomorphism} we then have a K-isomorphism $A^i \equiv L$ via $f_{b,\sigma^i} \mapsto b$ and an $L_i$-isomorphism $L \equiv A^{\sim 1}$ via
   $b \mapsto f^{\sim}_{b,\sigma^{i}}$. The composition of these maps clearly yields a $K$-isomorphism 
   $A^i \cong  A^{\sim 1}$. The following is then clear.
   \begin{rema}\label{crspn}
   With respect to the above isomorphism if an $L_i$-subspace $\mathcal W \le A^{\sim 1}$ has all its non-zero skew forms non-degenerate (or all its non-zero skew forms degenerate)  then the same is true for the corresponding (K-) subspace in $A^i$.  
\end{rema}\label{chng2Ai}

 \begin{lem}\label{Egnspcdecomp} 
     Let $n= 2^{\alpha}k$ where $ \alpha \geq 2$ and $k$ is odd.  Suppose that $L$ is a cyclic  extension of a field $K$ of degree $n$ with Galois group $\gal(L/K) = \langle \sigma \rangle$. Then the following hold.
     \begin{enumerate}
\item[(i)] For $1 \leq i \leq \alpha - 1 $ the subspace $ E_i:= \{b \in L : \sigma^{n/2^i}(b)= -b\} \le L$ has dimension $n/2^i$.

 \item[(ii)] Let $V_1 :=\{b \in L : \sigma^{k}(b) =  b\}$ and $V_2 :=\{b \in L : \sigma^{k}(b) =  -b\}$. Then $\dim(V_1) = \dim (V_2)= k$.
 \end{enumerate}
 \end{lem}
 \begin{proof}
     Let $1 \le i \le \alpha - 1$.  As the order of the automorphism $ \sigma^{n/2^i}$ is $ 2^i$ so the fixed field $L_{n/2^i}$ of $ \sigma^{n/2^i}$ has dimension $ n / 2^i$ over $K$.  We can view $\sigma^{n/2^i}$ as a $K$-linear map of $L$. By  the Dedekind independence theorem  the minimal polynomial of $\sigma^{n/2^i}$ is  $x^{2^i} - 1$. Let $ j_i\in L$ be an eigenvector of $ \sigma^{n/2^i}$ corresponding to the eigenvalue $ -1$. It is easily checked that the corresponding eigenspace is $E_i := j_i L_{n/2^i}$.  It follows that $\dim(E_i)= n/2^i$. The proof of (ii) is similar. 
 \end{proof}
 \begin{lem}\label{ degeneracy for E_i}
      Let $n= 2^{\alpha}k$ where $ \alpha \geq 2$ and $k$ is odd.  Suppose that $L$ is a cyclic  extension of a field $K$ of degree $n$ with Galois group $\gal(L/K) = \langle \sigma \rangle$. Then $\forall b_i \in E_i \setminus \{0\} $
      \begin{equation}
          N_{L/L_2}(b_i)= (-1)^{n/2^2} w_i^{2^i},
      \end{equation}
      where $w_i :=  b_i \sigma^{2}(b_i)\cdots \sigma^{n/2^{i} - 2} (b_i)$. Moreover, $f_{b_i,\sigma}$ is degenerate if and only if $\eta_i := \sigma(w_i)/w_i$ is a $ 2^i$-th root of unity in $L$ such that $\sigma(\eta_i)= - {\eta_i}^{-1}$. In particular, $f_{b_1,\sigma}$ is non-degenerate for all $b_1 \in E_1 \setminus \{0\}$.  
 \end{lem}
 \begin{proof}
      In view of the chain of inclusions \[ L \supset L_{n/2} \supset \cdots \supset L_{n/2^{i-1}} \supset E_i,\]
       we have for $ b_i \in E_i \setminus \{0\}$ \begin{align*}
           N_{L/L_2}(b_i)&= b_i \sigma^{2}(b_i)\cdots \sigma^{n-2}(b_i)\\
           &= \left( b_{i} \sigma^2(b_i) \cdots \sigma^{n/2 -2 }(b_i)\right) \left( \sigma^{n/2}(b_i) \sigma^{n/2 + 2}(b_i) \cdots \sigma^{n/2 + n/2 -2}(b_i)\right)\\
           &=\left(b_i \sigma^{2}(b_i)\cdots \sigma^{n/2 - 2} (b_i)\right)^{2}\\
           &=\left(b_i \sigma^{2}(b_i)\cdots \sigma^{n/4 - 2} (b_i)\right)^{2^2}\\
           &\ \ \ \ \ \ \ \ \ \ \  \ \ \ \ \ \ \ \ \  \ \ \vdots \\
           &=\left(b_i \sigma^{2}(b_i)\cdots \sigma^{n/2^{i-1} - 2} (b_i)\right)^{2^{i-1}}\\
           &=\left[\left(b_i \sigma^{2}(b_i)\cdots \sigma^{n/2^{i} - 2} (b_i)\right)\left(\sigma^{n/2^{i}}(b_i) \sigma^{n/2^{i} + 2}(b_i)\cdots \sigma^{n/2^{i} +n/{2^{i}} -2}(b_i) \right)\right]^{2^{i-1}}\\
           &= \left[\left(b_i \sigma^{2}(b_i)\cdots \sigma^{n/2^{i} - 2} (b_i)\right)\left((-b_i) (-\sigma^{2}(b_i))\cdots (-\sigma^{n/2^{i} - 2} (b_i))\right)\right]^{2^{i-1}}\\
           &= \left[(-1)^{n/2^{i+1}}\left(b_i \sigma^{2}(b_i)\cdots \sigma^{n/2^{i} - 2} (b_i)\right)^2\right]^{2^{i-1}}\\
           &= (-1)^{n/{2^2}}[b_i \sigma^{2}(b_i)\cdots \sigma^{n/2^{i} - 2} (b_i)]^{2^i}\\
          &=(-1)^{n/2^2} w_i^{2^i}.
       \end{align*} 
       Then 
       \begin{align*}
           \frac {N_{L/L_2}(\sigma(b_i))}{N_{L/L_2}(b_i)}= \left( \frac{(\sigma((-1)^{n/{2^2}}w_i))}{(-1)^{n/{2^2}}w_i}\right)^{2^i}
           = \left(\frac{\sigma(w_i)}{w_i}\right)^{2^i}= \eta_i^{2^i}.
       \end{align*}
       Set $\eta_i := \frac{\sigma(w_i)}{w_i}$.
        By Proposition \ref{GMdegen-crit}, $f_{b_i , \sigma}$ is degenerate if and only if $  \eta_i $ is a $ 2^{i}$-th  root of unity $ \eta_i $. Moreover, \[
           -w_i =\sigma^2(w_i)=\sigma(\eta_i w_i)=\sigma(\eta_i) \eta_i w_i,\]
       whence $ \sigma(\eta_i) \eta_i =-1$, that is,  $ \sigma(\eta_i) = -\eta_i^{-1}$. The last assertion in the theorem is now clear.
 \end{proof}
 \begin{lem}\label{degeneracy for V1 ,V2}
      Let $n= 2^{\alpha}k$ where $ \alpha \geq 2$ and $k$ is odd.  Suppose that $L$ is a cyclic extension of $K$ of degree $n$ with $\gal(L/K) = \langle \sigma \rangle$. Then  $\forall b \in V_1 \cup V_2 $, $f_{b,\sigma}$ is degenerate. 
 \end{lem}
 \begin{proof}
      \textsc{ Case I:} Let us first assume that $ k>1$. Then the field $V_1 = L_k$ has dimension $ k $ over $ K$. Again by Dedekind's independence theorem it follows that the minimal polynomial of $\sigma^k$ is $ x^{2^{\alpha}} -1$. Let $ j_\alpha $ be an eigenvector of $ \sigma^k$ corresponding to the eigenvalue $-1$ and it is easily checked that the corresponding eigenspace is $V_2 = j_\alpha L_k$.  Thus $\dim(V_1)=\dim(V_2)=k$. Note that $V_1$ and $V_2$ are $\sigma$-invariant. Again in view of the inclusions \[L \supset L_{n/2} \supset \cdots \supset L_{n/2^{\alpha-1}}=L_{2k} \supset L_k =V_1,\] we have $\forall b \in V_1 \setminus \{0\}$, 
       \begin{align*}
           N_{L/L_2}(b) &= b \sigma^{2}(b) \cdots \sigma^{n-2}(b)\\
                 &=\left(b \sigma^{2}(b) \cdots \sigma^{n/2^{\alpha-1}-2}(b)\right)^{2^{\alpha -1}}\\
                 &=\left( b \sigma^2(b) \cdots \sigma^{2k-2}(b)\right)^{2^{\alpha -1 }}\\
                 &=\left[\left(b \sigma^2(b)\cdots \sigma^{k-1}(b)\right)\left( \sigma^{k+1}(b) \cdots \sigma^{2k-2}(b)\right)\right]^{2^{\alpha - 1}}\\
                 &=\left[\left(b \sigma^2(b) \cdots \sigma^{k-1}(b)\right)\left( \sigma(b) \cdots \sigma^{k-2}(b)\right)\right]^{2^{\alpha -1}}\\
                 &=[b \sigma(b) \sigma^2(b)\cdots \sigma^{k-1}(b)]^{2^{\alpha-1}}\\
                 &= N_{L/L_2}(\sigma(b)).
       \end{align*}
        On the other hand in view of the inclusions \[L \supset L_{n/2} \supset \cdots \supset L_{n/2^{\alpha-1}}=L_{2k} \supset j_{\alpha}L_k =V_2,\] we have  $ \forall b \in V_2 \setminus \{0\}$, 
        \begin{align*}
           N_{L/L_2}(b) &= b \sigma^{2}(b) \cdots \sigma^{n-2}(b)\\
                 &=\left(b \sigma^{2}(b) \cdots \sigma^{n/2^{\alpha-1}-2}(b)\right)^{2^{\alpha -1}}\\
                 &=\left( b \sigma^2(b) \cdots \sigma^{2k-2}(b)\right)^{2^{\alpha -1 }}\\
                 &=\left[\left(b \sigma^2(b)\cdots \sigma^{k-1}(b)\right)\left( \sigma^{k+1}(b) \cdots \sigma^{2k-2}(b)\right)\right]^{2^{\alpha - 1}}\\
                 &=\left[\left(b \sigma^2(b) \cdots \sigma^{k-1}(b)\right) \left( (-\sigma(b)) \cdots (-\sigma^{k-2}(b)\right)\right]^{2^{\alpha -1}}\\
                 &=[b \sigma(b) \sigma^2(b)\cdots \sigma^{k-1}(b)]^{2^{\alpha-1}}\\
                 &= N_{L/L_2}(\sigma(b)).
                 \end{align*}
                 
Consequently $N_{L/L_{2}} (\sigma(b)/b) = 1$ and  
thus by Proposition \ref{GMdegen-crit} $ \forall b \in V_1 \cup V_2$, $ f_{b,\sigma}$ is degenerate.

\textsc{ Case II:}  We now assume that $ k = 1$  (thus $n = 2^\alpha$ and $ L_{2k} = L_2$).  Then $ V_1 := K$ and it is easily checked that  $ V_2  := j_{\alpha}K$, where $ j_{\alpha}$ is an eigenvector of $\sigma$ corresponding to the eigenvalue $-1$, 
       Thus $\dim(V_1)=\dim(V_2)=1$. Clearly if $ b \in L_2^{\times} $ then $ N_{L/L_2}(b) = b^{2^{\alpha-1}}$ and $N_{L/L_2}(\sigma(b)) = (\sigma(b))^{2^{\alpha -1}} $ as $L_2$ is $\sigma$-invariant. By definition if $b \in V_1 \cup V_2$ then $\sigma(b)= \pm b$ and in either case  
       
       \[ N_{L/L_{2}} \biggl (\frac{\sigma(b)}{b} \Biggr ) = { \Biggl (\frac{\sigma(b)}{b} \Biggr )}^{2^{\alpha -1}} = 1.\] Thus  by  Proposition \ref{GMdegen-crit} if $ b \in V_1 \cup V_2$, $ f_{b,\sigma}$ is degenerate.
 \end{proof}
\section{Proofs of Theorems A and B}
\subsection{Proof of Theorem A}
\begin{proof}
     Let $V := L_{k}$ and $0 \ne v \in V$. Clearly \[ \sigma^2(v), \sigma^4(v), \cdots, \sigma^{2k-2}(v) \in V. \]
It follows that \[ N_{L/L_2}(v) \in L_2 \cap V = L_2 \cap L_{k} = K. \] \label{ground field}
 By Proposition \ref{GMdegen-crit}   the skew-form $f_{v, \sigma}$ is degenerate and by Lemma \ref{rank of f b sigma for even case } it  has rank $n-2=2k-2$.

By Lemmas  \ref{GQ degeneracy condition} and \ref{existance a non degenarate} there exists a $j \in L$ such that $f_{j, \sigma}$ is non-degenerate. Then for $0 \ne v \in V$
\[ N_{L/L_2}(jv) = N_{L/L_2}(j)N_{L/L_2}(v) \not \in K.\]
It thus follows by  proposition \ref{GMdegen-crit} that all the nonzero skew-forms $ f_{b,\sigma}$ where $b$  lies in the subspace $U = jV$ (of dimension $k$) are non-degenerate. Clearly $U \cap V = \{0\}$ so $L= U \oplus V$. By Remark \ref{isomorphism} the subspace $U$ of $L $ corresponds to a subspace $\mathcal{U}$ of $\alt_K(L)$ with the same dimension defined by $\mathcal{U}:=\{f_{b,\sigma}: b \in U\}$. Similarly $V$ corresponds to $\mathcal{V} \le \alt_K(L)$ such that $\dim(V)=\dim(\mathcal{V})$. 
Then the decomposition \eqref{odd dcmposn for any field eq} follows.
\end{proof} 
\begin{coro}\label{dcompsn-Ai-for-odd}
    Let $K $ be a field and $n$ be even.  Suppose $L$ is a cyclic Galois extension of a field $K$ of degree $n$ with Galois group $\gal(L/K) = \langle \sigma \rangle$. If $ \ord(\sigma^{i}) \equiv 2~(\mo 4)$  and  $ \ord(\sigma^{i}) \ne 2 $ then   \[A^i  = \mathcal{U}_i \oplus \mathcal{V}_i,\] where  $ \mathcal{U}_i$ is an $n$-subspace  of dimension $n/2 $ and $\mathcal{V}_i$ is an $ (n-2n/ \ord(\sigma^{i}))$-subspace of dimension $n/2$. 
\end{coro}

\begin{proof}
     This follows from Theorem A, noting Remark \ref{crspn} and the fact (Lemma \ref{rank of f b sigma for even case }) that a skew form in $A^i$ is either non-degenerate or has rank equal to  $n - 2n/\ord(\sigma^i)$.
\end{proof}

Consequently we obtain the following. 

\begin{coro} \label{odd dcmpsn fr any field}
     Let $K$ be a field and $n = 2k$, where  $ k\geq 1$ is odd.  Let $L$ be any cyclic Galois extension of $K$ of degree $n$ with Galois group $G = \langle \sigma \rangle$. Then
     \begin{equation}
  \alt_K(L) =  B^1 \oplus \left(\bigoplus_{\substack{ \ord(\sigma^{i}) ~\equiv~ 0~ (\mo 2)\\\ord(\sigma^{i})\neq 2}}\left(\mathcal{U}_i \bigoplus \mathcal{V}_i \right)\right) \bigoplus \left( \bigoplus_{\ord(\sigma^{i}) ~\equiv~ 1~ (\mo 2)} A^i \right)   
\end{equation}  
\end{coro}
\begin{proof}
     Clear in view of Corollary \ref{dcompsn-Ai-for-odd}, Lemma  \ref{rank of f b sigma for odd case } as well as the decomposition \eqref{ dcmpstn for cyclic}.
\end{proof}
 \begin{rema} \label{n-subs-exists-in}
     Let $n= 2^{\alpha}k$ where $ \alpha \geq 1$ and $k$ is odd.  Suppose that $L$ is a cyclic  extension of a field $K$ of degree $n$ with Galois group $\gal(L/K) = \langle \sigma \rangle$. If $\ord(\sigma^i)$ is even then there always exists an $n$-subspace of dimension $n/2$ inside $A^i$.
    If $ \alpha = 1$ this follows from Corollary \ref{odd dcmpsn fr any field}.
     Otherwise if $ \alpha > 1$ then
     it follows from Lemma \ref{ degeneracy for E_i} that $\mathcal{E}_1: =\{f_{b,\sigma}: b\in E_1\} $ is the desired subspace for $A^1$.  The corresponding assertion for $A^i$ now follows in the light of Remark \ref{crspn}.
 \end{rema}
\subsection{ Proof of Theorem B}
\begin{proof}
    Firstly we will construct a cyclic extension $L$ of $ K $ such that $ i \notin L$ where $ i $ is a primitive $ 2^2$-th  root of unity. Let $p$ be a prime such that $ p \equiv 1 (\text{mod}~ n) $  and consider the cyclotomic extension $\mathbb{Q}(\eta_p)$ where $ \eta_p$ is a primitive $p$-th root of unity. As  is known  (e.g., \cite[Lemma 4]{GQ06}) it is possible to pick the prime $p$  as above such that $ \mathbb{Q}(\eta_p) \cap K(i) =\mathbb{Q} $. 
    Let $L$ be the unique 
 intermediate field  $\mathbb{Q} \subseteq L  \subseteq \mathbb{Q}(\eta_p)$  such that $[L:\mathbb{Q}] =n$. Clearly $ L \cap K(i) =\mathbb{Q}= L \cap K$. By a  well known fact (e.g., \cite[Chapter 6, Theorem 1.12]{SL1993}) the extensions $ LK(i)/K(i)$  and $ LK/K$ are Galois and  
  \[   \gal(LK(i)/K(i)) \cong \gal(L/{L \cap K(i)})= \gal(L/\mathbb{Q})= \gal(L/{L \cap K})\cong \gal(LK/K). \]
  If $i \in LK$ then by the last equation 
 \[ [LK:K]=[LK:K(i)][K(i):K]=[LK(i):K(i)][K(i):K], \] whence $[K(i):K]=1$ thus contradicting the hypothesis on $K$. Redefining $L:=LK$ yields the desired cyclic extension $L/K$ with degree $n$.\newline
    
Let $E_i :=  \{b \in L : \sigma^{n/2^i}(b)= -b\} \ (1 \le i \le \alpha -1)$. By Lemma \ref{Egnspcdecomp} we obtain $ L_{n/2^{i-1}}= L_{n/2^i} \oplus E_i$ and 
$L_{2k} = V_1 \oplus V_2$, where $V_1$ and $V_2$ denote the eigenspaces of $\sigma^k$ with respect to the eigenvalues $1$ and $-1$ respectively. Consequently, we obtain  
 \begin{equation} \label{A1decom}
           L = L_{n/2} \oplus  E_1 
           = L_{n/4} \oplus E_2 \oplus E_1
           = L_{2k} \oplus E_{\alpha - 1} \oplus \cdots \oplus E_1 = V_1 \oplus V_2 \oplus E_{\alpha - 1} \oplus \cdots \oplus E_1.
       \end{equation}      
Let  $\mathcal{E}_i$ be the subspace of $A^1$ corresponding to $E_i :=  \{b \in L : \sigma^{n/2^i}(b)= -b\}$ under the isomorphism of Remark \ref{isomorphism}, that is, $\mathcal{E}_i= \{ f_{b,\sigma}: b \in E_i\}$ ( $1 \leq i \leq \alpha -1$ ). 
    By our construction, the only $2^i$-th roots in $L$ are $\pm 1$. As $\sigma$ fixes both these roots, it follows from Lemma \ref{ degeneracy for E_i} that $\mathcal E_i$ is an $n$-subspace for all $i$ in the above range. 

    Similarly, let $\mathcal V_j$ correspond to the subspace $V_j$ of $L$. By Lemma \ref{degeneracy for V1 ,V2} the nonzero skew-forms in $\mathcal V_j$, where $j = 1,2$ are degenerate whence these are $(n - 2)$-spaces by Lemma \ref{rank of f b sigma for even case }.  The required decomposition \eqref{decmpsn A1 algebraic} is  now immediate from \eqref{A1decom}. 
   \newline
                 \end{proof}
                 \begin{coro}\label{decmpsn Ai algebraic}    
In the situation of Theorem B if $ \ord(\sigma^{i}) \equiv 0~(\mo 4)$, say $\ord(\sigma^{i})= 2^{\beta}k' \ (\beta \ge 2)$ then   
   \begin{equation}
A^{i} = \mathcal{V}_1^i \oplus \mathcal{V}_2^i \oplus \mathcal{E}_1^i \oplus \cdots \mathcal{E}_{\beta - 1}^i,
   \end{equation}
    where
    \begin{itemize}
        \item[(i)] $\mathcal{E}_{k}^{i}$ is an  $n$-subspace of dimension $n/2^i$ for   $  1 \leq k \leq \beta - 1$,
        \item
        [(ii)]  $\mathcal{V}_j^i$ is an $(n-2)$-subspace of dimension $ k'n/{\ord(\sigma^{i})} $ for  $  1 \leq j \leq 2$.
        \end{itemize}
                 \end{coro}
                 \begin{proof}
                      This follows from proof of Theorem B, noting Remark \ref{chng2Ai} and the fact (Lemma \ref{rank of f b sigma for even case }) that a skew form in $A^i$ is either non-degenerate or has rank equal to  $n - 2n/\ord(\sigma^i)$.
                 \end{proof}
                 \begin{coro}\label{algebraic dcmpsn Alt L}
                     In the situation of Theorem B there is direct-decomposition
                     \begin{equation}
                      \begin{aligned}[b]
            \alt_K(L) = &  B^1 \bigoplus \left(\bigoplus_{\substack{ \ord(\sigma^{i}) ~\equiv~ 2~ (\mo 4)\\\ord(\sigma^{i})\neq 2}} \left(\mathcal{U}_i \bigoplus \mathcal{V}_i \right)\right) \bigoplus \left( \bigoplus_{\ord(\sigma^{i}) ~\equiv~ 1~ (\mo 2)} A^i \right)  \\ &\bigoplus_{ \ord(\sigma^{i}) ~\equiv~ 0~ (\mo 4)}\left(\mathcal{V}_1^i \bigoplus \mathcal{V}_2^i \bigoplus \mathcal{E}_{\beta - 1}^i \bigoplus \cdots \bigoplus\mathcal{E}_1^i\right) 
            \end{aligned}
            \end{equation}
                 \end{coro}
                 \begin{proof}
                    Using Corollaries \ref{dcompsn-Ai-for-odd}, \ref{decmpsn Ai algebraic}   and Lemma \ref{rank of f b sigma for odd case }  as well as the decomposition \eqref{ dcmpstn for cyclic}, we can deduce the  required decomposition. 
                 \end{proof}
                 \begin{rema}\label{theorem B rema}
                     As its proof shows, Theorem B as well as its corollaries  remain valid for an arbitrary cyclic extension $L/K$ of degree $n = 2^{\alpha}k$ ($\alpha \ge 2$) such that $-1$ is not a square in $L$. Similarly, let $K$ be a field such that $f(X): = X^{4} + 1 $ is irreducible in $K[X]$ (it is not difficult to show that $K$ has this property if and only if none of $ -1 ,2$ and $-2$ is a square in $K$). Then Theorem $B$ holds true for any cyclic extension $L/K$ of degree $n = 2^\alpha k$. Indeed, if $\eta_i$ is a $2^i$-root of unity for $i \ge 1$ then the conditions $-1 \not \in K^2$ and  $\sigma(\eta_i) = -\eta_i^{-1}$ mean that $\eta_i \not \in \{-\pm 1, \pm i \}$, where $i$ denotes a primitive $4$-th root of unity in $L$. Thus $\eta$ must have order $2^s$ where $s \ge 3$. Since $\eta \in L_2$ this would mean that $L_2$ contains an element of order $8$ and thus a root of $f$ implying $f$ has a quadratic factor in $K[X]$. 
                     
     
\end{rema}
\section{Proofs of Theorems C and D}
\subsection{Proof of Theorem C}
\begin{proof}
   Let $ E_i$ ( $1 \leq i \leq \alpha -1$ ) and $ V_j$ ( $1 \leq j \leq 2$ ) be  as in Lemma \ref{Egnspcdecomp}. As in the proof of Theorem B, we have \begin{equation*} 
           L =  V_1 \oplus V_2 \oplus E_{\alpha - 1} \oplus \cdots \oplus E_1.
       \end{equation*}  
   By the hypothesis $-1$ is not a square in $K$ from which it easily follows  that $a \geq 2$. Let $ w_i$ and $ \eta_i$ be as in Lemma \ref{ degeneracy for E_i}. Note that $ \sigma_f^2(w_i) = -w_i$ and thus  $w_i^2\in L_2$ but $ w_i \notin L_2$. Consequently $ w_i^{2(q^2 -1)} = 1$ and $w_i^{(q^2 -1)}= -1$.   Since $\sigma_f(w_i) = w_i^q$ hence $ \eta_i = w_i^{q-1}$.  
       It follows that $ \eta_i$ is a $2(q+1)$-th root of unity but not a $ (q+1)$-th root of unity.
       \newline
       (1) Suppose $ \alpha \leq a+1$. Since $1 \leq i\leq \alpha -1$ therefore $1\leq i \leq a$. Again  by Lemma \ref{ degeneracy for E_i},   $f_{b_i,\sigma_f}$ is degenerate if and only if $ \eta_i $ is a $2^{i}$-th root of unity. Since $ i \leq  a$, this would mean that $\eta_i^{q+1}= \eta_i^{2^a l}=1$, a contradiction. Let $\mathcal{E}_i$ be the subspace of $ A^1$ corresponding to $ E_i$ under the isomorphism of Remark \ref{crspn}. It follows that  $\mathcal{E}_i$ is an  $n$-subspace of dimension $n/2^i$.
       \newline
       (2) Suppose $\alpha > a+1$. Pick $ i \in [1,\alpha -1] $.  If $ 1 \leq i\leq a$ it follows from part (1) above that $ E_i$ is an $n$-subspace for $1 \leq i \leq a$. So we assume that $i \geq a+1$. 
        By the hypothesis $l = 1$, whence   $ \eta_i^{2^{a+1}}= \eta_i^{2(q+1)}=1$.  It follows that if $ a+1 \leq i \leq \alpha -1 $  then $  \eta_i^{2^i}=1$.  Thus in view of Lemma \ref{ degeneracy for E_i} all the skew-forms in $ \mathcal{E}_i $ are degenerate and in this case by Lemma \ref{rank of f b sigma for even case }, $\mathcal{E}_i$ is an $(n-2)$-subspace.
        \newline
          Similarly let $ \mathcal{V}_j$ be the subspace of $ A^1$ corresponding to $V_j$. Then by Lemmas \ref{degeneracy for V1 ,V2} and  \ref{rank of f b sigma for even case }, $\mathcal{V}_j$ is an $(n-2)$-subspace.
        \end{proof}
        \begin{rema}
        In Theorem C when $ \alpha > a+1$ and $ l > 1$ then $ \mathcal{E}_i$ is neither an $ n$-subspace nor an $( n-2)$-subspace for $a+1 \leq i \leq \alpha -1$. Indeed, by the definition  of $E_i$
       \begin{align*}
           E_{i} = \{ b\in L : \sigma_f^{n/2^i}(b) = -b \}
           = \{ b\in L : b^{q^{n/2^i} - 1} = -1 \}.
       \end{align*}
       Let $C:= \{b\in L^{\times}: b^{2( q^{n/2^i}-1)}=1\}$. Then $ C $ is a cyclic subgroup of $L^{\times}$. Clearly, $ C = L_{n/2^i}^{\times} \bigcupdot ( E_{i}
       \setminus \{0\} )$. Let $u$ be a generator of  $ C$. It is clear that  $b_i = u^s \in E_i$ if and only if $s$ is  odd. We claim that $f_{b_i, \sigma_f}$ is degenerate if and only if $s$ is an odd multiple of $l$.           
       Indeed,  let $ w_i$ and $ \eta_i$ be as in Lemma \ref{ degeneracy for E_i}. 
Then
       \[
           w_i = b_i \sigma_f^2(b_i) \cdots \sigma_f^{n/2^i  -2}(b_i)
           = b_i b_i^{q^2}\cdots b_i^{q^{n/2^i -2}}
            = b_i^{\frac{q^{n/2^{i}} -1}{q^{2} -1}},
            \]
             and \[ \eta_i =  w_i^{q-1} = b_i^{\frac{q^{n/2^{i}} -1}{q +1}} = b_i^t,\] where $ t:= {\frac{q^{n/2^{i}} -1}{q +1}} $.
              By Lemma \ref{ degeneracy for E_i}, $f_{b_i, \sigma}$ is degenerate if and only if $\eta_i^{2^i}=1$. Now from the proof of Theorem C, $\eta_i^{2^{a+1}l}=\eta_i^{2(q+1)}=1$ and $\eta_i^{2^{a}l}=\eta^{(q+1)}\neq1$. Consequently $ f_{b_i,\sigma_f}$ is degenerate if and only if $ \eta_i$ is a  primitive $ 2^{a+1}$-th  root of unity, that is,  if and only if, 
              \begin{equation}
               2^{a+ 1} =  \ord(\eta_i) = \ord(u^{st})
                  = \frac{\ord(u)}{\gc(\ord(u), st)}\\
                  = \frac{2(q+1)t}{\gc(2(q+1)t,st)}
                  = \frac{2^{a+1}l}{\gc (2^{a+1}l, s)}, 
              \end{equation}
              or, $\gcd(2^{a+1}l,s) = l$.  In other words,  for $b_i = u^s  \in E_i $, $ f_{b_i,\sigma_f}$ is degenerate if and only if  $s$ is an odd multiple of $l$. Thus, for example,  $ f_{u^l, \sigma_f}$ is degenerate while  $ f_{u, \sigma_f}$ is non-degenerate.
               \end{rema}           
    \subsection{Proof of Theorem D}
  \begin{proof}
      By \cite[Proposition 5.4.11]{pg} for every  $ n $ there exists exactly one unramified extension  $ L $ of $ K=\mathbb{Q}_p $ of degree $ n $ obtained by adjoining a  primitive $ (p^n -1)$-th  root of unity, say $ \theta$. Moreover  according to \cite[Corollary 2]{number}, the extension $L/K$ constitutes a cyclic  extension such that $\gal(L/K) = \langle\sigma \rangle$ where  $\sigma$ is 
 defined by $\sigma(\theta) = \theta^{p} $. Since $ -1$ is not a square in $K$ so $p=2^al-1 \equiv 3~ (\mo 4 )$ by \cite[Proposition 3.4.2]{pg} and thus $a\geq 2$. 

  Let $ E_{i}:= \{ b\in L : \sigma^{\frac{n}{2^i}}(b) = -b\}$ where $1 \leq i \leq \alpha - 1$. The  hypothesis $ \alpha \leq a+1$ means that $1 \leq i \leq a$. Let $w_i,\eta_i$ be as in Lemma \ref{ degeneracy for E_i}. 
Again by Lemma \ref{ degeneracy for E_i}, $f_{b_i,\sigma} \ (b_i \in E_i) $ is degenerate if and only if $\eta_i$ is $ 2^i$-th  root of unity such that  $\sigma(\eta_i)=-\eta_i^{-1}$. As $2^i \mid 2^a \mid p + 1 \mid p^n -1$, this would mean that $\langle \eta_i \rangle  \le \langle \theta \rangle$ and consequently,  $\sigma(\eta_i)= \eta_i^p$. But then \[\sigma(\eta_i) \eta_i = \eta_i^{p+1} = \eta_i^{2^{a}l}=1.\]  It follows that  $ f_{b_i ,\sigma } $ is non-degenerate. Hence $ \mathcal{E}_i$ is an $n$-subspace, where $ \mathcal{E}_i$ is the subspace of $ A^1$ corresponding to $E_i$. \newline  Similarly let $ \mathcal{V}_j$ ( $1 \leq j \leq 2$ ) be the subspace of $ A^1$ corresponding to $V_j$. Then by Lemmas \ref{degeneracy for V1 ,V2} and  \ref{rank of f b sigma for even case }, $\mathcal{V}_j$ is an $(n-2)$-subspace. The 
 theorem now follows.
\begin{rema}
    In the situation of Theorem D for $p=2$ the decomposition \eqref{algebraic dcmpsn Alt L} holds true in view of Remark \ref{theorem B rema}.
\end{rema}

  \end{proof}
  \section{ A $3$-dimensional $4$-subspace in $\alt_4(\mathbb Q)$}
  Let $K := \mathbb{Q}$ and $L$ be the cyclotomic field $\mathbb{Q} (\eta)$ where $\eta$ is a primitive $5$-th root of unity in $\mathbb C$. Then $L/K$ is a cyclic extension of degree $4$. We will show that  the maximum dimension of a $4$-subspace inside $ A^1 $ is $ 3$. Let $ b  = x + y\eta +z\eta^{2} + w\eta^{3} \in L $, where $ x,y,z,w \in \mathbb Q$. We take the automorphism $ \sigma$ defined by $ \sigma(\eta)= \eta^3$ as a generator of $\gal(L/K)$. Using  the theory of Gauss periods we may find the basis, namely, $ \{ 1, \eta^2 +\eta^3 \} $ for $L_{2}/\mathbb{Q}$. By  Proposition \ref{GMdegen-crit},  $ f_{b,\sigma} $ is degenerate if and only if   $ N_{L/L_{2}}(b)  \in \mathbb{Q}$, that is, the coefficient of $ \eta^2 +\eta^3 $ in $ N_{L/L_{2}}(b) $  is zero. It is straight-forward to check that this coefficient is $ -xy+ xz +xw - yz + yw - zw $.
  In this situation we thus obtain the following.
  \begin{prop}
      \label{totally isotropic}
	The maximum dimension of a $ 4 $-subspace inside $ A^{1} $  equals to a maximum dimension of a totally anisotropic subspace of $ L $  with respect to the 
 following quadratic form  \[ \mathcal{Q}(x,y,z,w)=xy -xz - xw+ yz - yw + zw .\] 
  \end{prop}
  \begin{proof}
    Clear.
\end{proof}
\begin{thm}

 \textbf{(Legendre's Theorem)}\label{legendre}\emph{(\cite[Theorem 1, Chapter 5]{EG})} Suppose $ a,b,c \in \mathbb{Z}$ are such that $ abc$ is a non-zero square-free integer. Then the equation $ a X^2 + b Y^2 + c Z^2 = 0 $ has a non-trivial $Z$-solution if and only if $(i)$ $a,b,c$ do not all have the same sign; $(iia)$ $-bc $ is a square modulo $|{a}|$, $(iib)$ $-ac $ is a square modulo $|{b}|$ and $(iic)$ $-ab $ is a square modulo $|{c}|$.
\end{thm}
\begin{thm}\label{counter}
    The maximum dimension of a $ 4 $-subspace in $ A^{1} $ is $ 3 $.
\end{thm}
\begin{proof}
	 Let $ U $ be the $\mathbb Q$-subspace of $ L $ spanned by $ \{\eta + \eta^{2} , -1 + \eta^{3},1+\eta\} $. Let $ b = c_{1} (\eta + \eta^{2}) + c_2(-1 + \eta^{3}) + c_3(1+\eta) $. We claim that $\mathcal W :=   \{f_{b,\sigma} : b\in U \} \le A^1$ is the desired $4$-subspace. Indeed, according to  proposition \ref{totally isotropic} we need to show that the quadratic form  $$ \mathcal{Q}(c_1,c_2,c_3)= c_{1}^2  +c_{2}^{2} + c_{3}^{2} +c_1 c_{3} - 3 c_{2}c_{3} $$ has no non-trivial integer solution. It can be checked that  $\mathcal{Q}$  reduces to it's diagonal form  $$ \mathcal{Q'}= c_{1}^{2} + c_{2}^{2} - 6 c_{3}^{2}  .$$ To complete the proof, it suffices to show that $\mathcal{Q'}$ has no non-trivial integer solutions. Based on Theorem \ref{legendre} it is evident  that $\mathcal{Q'}$ has no non-trivial integer solutions since $-ab= -1$ is not square modulo $|c|=6$.
	\end{proof}
  \section{Conclusion}
Eigenspaces of the elements of the Galois group yield constant rank subspaces in $\alt_K(L)$.  We can always find an $ n $-subspace of dimension $ n/2$ in $A^i$ for an arbitrary field $K$ (Remark \ref{n-subs-exists-in}).  However, this may not be the maximum possible dimension of an $n$-subspace in $ A^1$ (as is evident from the example in Section 5) unless $n =2k$ with $k$ odd (Theorem A) or $ K$ is finite (or more generally $C^1$~{\emph{\cite[Lemma 3]{GQ06}}} ). Moreover unless $K$ is finite it is not clear that we get an $n$-subspace of maximum dimension of $\alt_{n}(K)$ in this way. The question of the maximum dimension of an $ n$-subspace in $\alt_n(K)$ is closely related to other invariants for skew-forms including $d(K,n,1)$ and  $s_n(K)$ defined in  \cite{BGH1987} and \cite{GQ06} respectively. In particular, it is unknown to the authors if there is a $6$-subspace in $\alt_6(\mathbb{Q})$ of dimension four.
   \section*{Acknowledgements}
The second author gratefully acknowledges support from an NBHM research award. 

\end{document}